\documentclass[reqno]{amsart}

\usepackage{amssymb,latexsym,amsmath,amsthm,amsbsy}
\usepackage[mathscr]{eucal}
\usepackage{mathrsfs}
\usepackage{mathtools}
\usepackage{array,booktabs,tabularx}
\usepackage{enumitem,cite}
\usepackage{microtype}
\usepackage{xurl}
\usepackage[hidelinks]{hyperref}

\makeatletter
\@namedef{subjclassname@2020}{\textup{2020} Mathematics Subject Classification}
\makeatother

\theoremstyle{plain}
\newtheorem{thm}{Theorem}[section]
\newtheorem{lem}[thm]{Lemma}
\newtheorem{cor}[thm]{Corollary}
\newtheorem{pro}[thm]{Proposition}
\newtheorem{example}[thm]{Example}
\theoremstyle{definition}
\newtheorem{defn}[thm]{Definition}
\newtheorem{remark}[thm]{Remark}

\setlist[enumerate]{label=\textup{(\roman*)},leftmargin=*,itemsep=2pt,topsep=4pt}
\allowdisplaybreaks[1]

\newcommand{\N}{\mathbb N_0}
\newcommand{\R}{\mathbb R}
\newcommand{\Q}{\mathbb Q}
\newcommand{\Z}{\mathbb Z}
\newcommand{\Smax}{S}

\newcommand{\Fvar}{\mathcal F}
\newcommand{\Uvar}{\mathcal U}

\newcommand{\V}{\mathsf{V}}
\newcommand{\Id}{\operatorname{Id}}
\newcommand{\Mat}{\mathbf{M}}
\newcommand{\UT}{\mathbf{UT}}
\newcommand{\sgn}{\operatorname{sgn}}

\newcommand{\Syn}{\mathfrak S}
\newcommand{\Alg}{\mathfrak A}
\newcommand{\Pal}{\mathcal P}
\newcommand{\eps}{\varepsilon}

\newcommand{\Alphabet}{\mathcal A}
\newcommand{\PP}{\mathbb P_{\mathrm{odd}}}

\newcommand{\SR}{S_{\mathbb R_{\geq0}}}
\newcommand{\M}[1]{\Mat_{#1}(S)}

\hypersetup{
 pdftitle={The finite basis problem and subvariety intervals for max-plus matrix semirings},
 pdfauthor={Jun Jiao and Mengya Yue},
 pdfsubject={Algebraic countermodels and continuum subvariety intervals},
 pdfkeywords={semiring, max-plus matrix, variety, identity, finite basis problem}
}

\begin{document}
\title[Max-plus matrix semirings]
{The finite basis problem and subvariety intervals for max-plus matrix semirings}

\author{Jun Jiao}

\author{Mengya Yue}
\date{}
\subjclass[2020]{16Y60, 03C05, 08B15}
\keywords{semiring, max-plus matrix, upper triangular matrix, variety, identity,
finite basis problem, subvariety interval}

\begin{abstract}
Let $S=(\N,\max,+)$ be the max-plus algebra of the nonnegative integers,
and let $\overline S$ be obtained from $S$ by adjoining $-\infty$ as a
semiring zero. We first prove that the varieties generated by the full
matrix semirings $\Mat_n(S)$ and by the upper triangular matrix
semirings $\UT_n(\overline S)$ form strictly ascending chains. Every
interval between distinct members of either chain has cardinality
$2^{\aleph_0}$ and contains a chain and an antichain of that cardinality.
The proof uses finite syntactic semirings associated with prime-cycle
languages and an independent family of identities.

We then construct finite commutative ai-semirings from languages of
monomials closed under divisibility and prove their bounded-variable
property by cancellation, explicit integer evaluations and independence
in a free abelian group. A loop-block identity for upper triangular
matrices and dimension-independent cubic-padding identities for full
matrices show that every variety in each of the above intervals has
infinite axiomatic rank. In fact, every variety between $\V(S)$ and
the variety generated by all full matrix dimensions has infinite
axiomatic rank. These varieties are therefore nonfinitely based, also
relative to the ai-semiring axioms.

Finally, for every fixed full matrix dimension, we prove equational
equivalence over all six positive or nonnegative integer, rational and
real coefficient domains. The original convex-ideal countermodels and
their full-matrix transfer proof are recalled separately for comparison;
they are not used in the algebraic proofs. Throughout, the signature
consists of two binary operations and contains no constant symbols.
\end{abstract}

\maketitle
\markboth{JUN JIAO AND Mengya Yue}{MAX-PLUS MATRIX SEMIRINGS}

\section{Introduction and preliminaries}

By an \emph{additively idempotent semiring} (ai-semiring for short) we
mean an algebra $(A,\oplus,\cdot)$ such that the additive reduct
$(A,\oplus)$ is a commutative idempotent semigroup, the multiplicative
reduct $(A,\cdot)$ is a semigroup, and the distributive laws
\[
 x(y\oplus z)\approx xy\oplus xz,
 \qquad (x\oplus y)z\approx xz\oplus yz
\]
hold. Thus, together with these two distributive laws, the ai-semiring
axioms are
\[
 \begin{gathered}
 (x\oplus y)\oplus z\approx x\oplus(y\oplus z),\qquad
 x\oplus y\approx y\oplus x,\qquad x\oplus x\approx x,\\
 (xy)z\approx x(yz).
 \end{gathered}
\]
We usually omit the multiplication symbol. Define an order on $A$ by
\[
 a\leq b\quad\Longleftrightarrow\quad a\oplus b=b.
\]
This order is compatible with both operations. Whenever an order on an
ai-semiring is mentioned, it refers to this order. For terms, the
inequality $u\leq v$ abbreviates the identity $u\oplus v\approx v$.

For an ai-semiring $A$, let $\V(A)$ denote the variety generated by $A$,
and let $\Id(A)$ denote the set of all identities satisfied by $A$.
A variety is a class closed under homomorphic images, subalgebras and
arbitrary direct products. By Birkhoff's theorem, these are exactly the
classes defined by identities; see \cite{BS}. Consequently, $A$ and
$\V(A)$ have the same equational theory.
A variety is \emph{finitely based} if it can be defined by finitely many
identities; otherwise, it is \emph{nonfinitely based}. An algebra is
finitely based if its generated variety is finitely based. All varieties
considered here are in the signature consisting of the two binary
operations. In particular, the existence of a zero or an identity element
in an algebra does not make that element a named constant.

The present paper concerns the max-plus algebras
\[
 S=(\N,\oplus,\odot),\qquad
 \overline S=(\N\cup\{-\infty\},\oplus,\odot),
\]
where $\N=\{0,1,2,\ldots\}$ and
\[
 a\oplus b=\max\{a,b\},\qquad a\odot b=a+b.
\]
Here the addition on the right is ordinary arithmetic addition, and
$a+(-\infty)=(-\infty)+a=-\infty$. The numerical value $0$ is a
multiplicative identity in both algebras. The element $-\infty$ is the
additive identity and multiplicative zero of $\overline S$.

Let $n\geq1$ be an integer. We denote by $\Mat_n(S)$ the ai-semiring of
all $n\times n$ matrices over $S$, and by $\UT_n(\overline S)$ the
ai-semiring of all upper triangular matrices over $\overline S$.
Thus an entry below the diagonal of a matrix in $\UT_n(\overline S)$
is $-\infty$, and entries on or above the diagonal may also equal
$-\infty$. In both cases, the matrix operations are given by
\[
 (A\oplus B)_{ij}=\max\{A_{ij},B_{ij}\},\qquad
 (AB)_{ij}=\max_{1\leq k\leq n}(A_{ik}+B_{kj}).
\]
For convenience, put
\begin{equation}\label{eq:varieties}
 \Fvar_n=\V(\Mat_n(S)),\qquad
 \Uvar_n=\V(\UT_n(\overline S)).
\end{equation}

Aceto, \'Esik and Ing\'olfsd\'ottir \cite{AEI,AEIreport} proved that
$S$ has no finite equational basis. The finite basis problem remains subtle even for small finite
ai-semirings; see the flat extensions in \cite{RJZL} and the
four-element examples in \cite{YRG}. 
The estimate uses nonnegative finite entries and does not extend directly
to upper triangular matrices after $-\infty$ has been adjoined.
Moreover, an embedding of $S$ into a matrix semiring does not, by itself,
transfer the nonfinite basis property. A separate family of matrix
identities and countermodels is required.
In the full-matrix construction, a variable $z$ is placed on either
side of each $x_j$ on the left, and is replaced by $z^3$ on the right.
The cube gives a uniform entrywise lower bound without requiring the
matrices to commute. Substitution of an actual multiplicative identity
for $z$ in a scalar countermodel removes the padding.

The multiplicative monoid of two-by-two upper triangular tropical
matrices was studied by Chen, Hu, Luo and Sapir \cite{CHLS}.
Here both semiring operations are retained. The distinction between
finite nonnegative coefficients and coefficients that include negative
numbers or $-\infty$ is essential for the cubic estimate.
The coefficient-domain results below identify the full matrix theories
over six positive or nonnegative domains; they do not assert such an
identification with all real tropical coefficients.

Another question concerns the subvariety intervals between different
matrix dimensions. Related constructions for Boolean and flat semirings
appear in \cite{Boolean,Volkov,S7note,T2note}. 
of finite commutative countermodels. Their bounded-variable property
will be proved directly, without the convex-ideal construction of
\cite{AEI,AEIreport} or a geometric characterization of scalar identities.

Let $\mathcal X$ and $\mathcal Y$ be varieties with
$\mathcal X\subseteq\mathcal Y$. As usual, the interval
$[\mathcal X,\mathcal Y]$ consists of all varieties $\mathcal Z$ such that
$\mathcal X\subseteq\mathcal Z\subseteq\mathcal Y$.
The matrix-variety chains considered in this paper are strictly
ascending. Accordingly, the adjacent full-matrix interval is
$[\V(\Mat_n(S)),\V(\Mat_{n+1}(S))]$.

For a class $\mathcal C$ of algebras, let $\Id_r(\mathcal C)$ be the set
of its identities involving at most $r$ distinct variables. The class
has \emph{infinite axiomatic rank} if $\Id_r(\mathcal C)$ does not imply
all its identities for any positive integer $r$. Since every finite
set of identities has a bound on its number of variables, infinite
axiomatic rank implies the nonfinite basis property.
We now state the main result.

\begin{thm}\label{thm:main}
Let $n$ and $m$ be positive integers with $n<m$. Then
\begin{equation}\label{eq:maincard}
 \begin{gathered}
 \Fvar_n\subsetneq\Fvar_m,\qquad \Uvar_n\subsetneq\Uvar_m,\\
 |[\Fvar_n,\Fvar_m]|=|[\Uvar_n,\Uvar_m]|=2^{\aleph_0}.
 \end{gathered}
\end{equation}
For every $n\geq1$, the intervals $[\Fvar_n,\Fvar_{n+1}]$ and
$[\Uvar_n,\Uvar_{n+1}]$ each contain a join-preserving order-embedded
copy of $\Pal(\PP)$, where $\PP$ is the set of odd primes.
Consequently, each interval in \eqref{eq:maincard} contains a chain
and an antichain of cardinality $2^{\aleph_0}$.

Every variety in every interval in \eqref{eq:maincard} has infinite
axiomatic rank. In particular, every such variety is nonfinitely based,
both absolutely and relative to the ai-semiring axioms.
\end{thm}

The two families of finite algebras used below have different roles.
Finite syntactic semirings of prime-cycle languages distinguish the
matrix dimensions and the selected primes. Finite commutative semirings
of divisor languages distinguish bounded-variable identities from
identities valid in the matrix semirings. The latter construction uses
only divisibility, integer evaluations, cancellation and independence
in a free abelian group.

The paper is organized as follows. In Section~\ref{sec:matrix}, we
record matrix products, scalar identities and embeddings, and then
establish the identities that distinguish successive dimensions.
Sections~\ref{sec:syntactic} and \ref{sec:cycles} construct finite
syntactic semirings and independent identities. Section~\ref{sec:continuum}
determines the cardinality of the matrix-variety intervals.
Section~\ref{sec:algebra} gives the finite commutative countermodels
and proves their bounded-variable property. Sections~\ref{sec:triangularnfb}
and \ref{sec:fullnfb} apply them to upper triangular and full matrices,
respectively. Section~\ref{sec:domains} treats coefficient domains.

\section{Matrix embeddings and standard identities}\label{sec:matrix}

In this section, we first record matrix products and the scalar cyclic
identities. We then establish embeddings and use a standard identity
to distinguish successive dimensions.

\subsection{Matrix products and scalar identities}
For any nonempty $D\subseteq\R_{\geq0}$ closed under maximum and
ordinary addition, write $S_D=(D,\max,+)$. The order on $\Mat_n(S_D)$
is the entrywise order. Full matrices with finite entries in $\N$
need not have an ambient multiplicative identity when $n>1$, since
the usual max-plus identity matrix has off-diagonal entries $-\infty$.
No ambient identity is required in the arguments for these full
matrix algebras.

\begin{lem}\label{src:lem:paths}
Let $T_1,\ldots,T_\ell$ be $n\times n$ max-plus matrices with finite
real entries. With $k_0=a$ and $k_\ell=b$,
\begin{equation}\label{src:eq:path}
(T_1\cdots T_\ell)_{ab}
 =\max_{k_1,\ldots,k_{\ell-1}}
   \sum_{h=1}^{\ell}(T_h)_{k_{h-1}k_h}.
\end{equation}
If every entry of $T_h$ is at most $\alpha_h$, then every entry of
their product is at most $\sum_h\alpha_h$. If every entry of $T_h$
is at least $\beta_h$, then every entry of their product is at least
$\sum_h\beta_h$.
\end{lem}

\begin{proof}
For $\ell=1$, the formula has no intermediate indices. For $\ell=2$,
it is the definition of matrix multiplication. If it holds for
$\ell-1$, multiplying the first $\ell-1$ factors by $T_\ell$ adds one
intermediate index and one summand; the two finite maxima combine
to give \eqref{src:eq:path}. This proves the formula by induction.
Under the upper bounds, each path sum is at most
$\sum_h\alpha_h$, and hence so is their maximum. Under the lower
bounds, any choice of intermediate indices has path sum at least
$\sum_h\beta_h$, so the maximum has the same lower bound.
\end{proof}

\begin{lem}\label{src:lem:embedding}
For every $n\geq1$, $\Smax$ embeds into $\M n$. Consequently,
\begin{equation}\label{src:eq:idinclusion}
\Id(\M n)\subseteq\Id(\Smax),
\qquad
\V(\Smax)\subseteq\V(\M n).
\end{equation}
\end{lem}

\begin{proof}
Let $C_n(a)$ be the matrix all of whose entries are $a$. For
$a,b\in\N$,
\[
C_n(a)\oplus C_n(b)=C_n(\max\{a,b\}),
\qquad C_n(a)C_n(b)=C_n(a+b).
\]
Indeed, every candidate in every entry of the latter product is the
same number $a+b$. The injective map $a\mapsto C_n(a)$ therefore
preserves both operations. An identity holding for all matrix
assignments holds, in particular, on these constant matrices; this
gives the identity inclusion. The variety inclusion follows from
the embedding.
\end{proof}

For $m\geq2$, let $i^+$ denote the cyclic successor of $i$ in
$\{1,\ldots,m\}$. Thus $m^+=1$. Define
\begin{equation}\label{eq:scalarwords}
p_m=x_1x_2\cdots x_m,\qquad
q_{m,i}=\prod_{j=1}^{m}x_j^{\,1+\mathbf1_{j=i}-\mathbf1_{j=i^+}},
\qquad
q_m=\bigoplus_{i=1}^{m}q_{m,i}.
\end{equation}
The product defining $q_{m,i}$ is taken in increasing order of $j$.
An exponent zero means that the factor is omitted; it does not
introduce a constant or an empty word. Every $q_{m,i}$ has total
length $m$. We use the scalar identities
\begin{equation}\label{src:eq:em}
e_m:\qquad p_m\oplus q_m\approx q_m.
\end{equation}

\begin{lem}\label{src:lem:scalarvalid}
The algebra $\Smax$ satisfies $e_m$ for every $m\geq2$.
\end{lem}

\begin{proof}
Under a scalar assignment $x_j=a_j$, put $h=\sum_ja_j$.
The value of $p_m$ is $h$, while the value of $q_{m,i}$ is
$h+a_i-a_{i^+}$. At least one index satisfies $a_i\geq a_{i^+}$:
otherwise the strict inequalities would continue around the cycle
and imply $a_1<a_1$. Thus the maximum of the right-hand monomials
is at least $h$.
\end{proof}

\begin{example}\label{src:ex:scalarfails}
The identity $e_2$, namely $xy\leq x^2\oplus y^2$, fails in
$\M2$. Let
\[
A=\begin{pmatrix}1&0\\0&1\end{pmatrix},
\qquad
B=\begin{pmatrix}0&1\\1&0\end{pmatrix}.
\]
Then
\[
AB=\begin{pmatrix}1&2\\2&1\end{pmatrix},
\qquad
A^2=B^2=\begin{pmatrix}2&1\\1&2\end{pmatrix}.
\]
The off-diagonal entries violate the required inequality.
Thus the scalar separating identities themselves cannot be used
as matrix-valid separating identities.
\end{example}

\subsection{Dimension embeddings and standard identities}
We first recall the following embedding.

\begin{lem}\label{lem:embedding}
Let $A$ be any ai-semiring, let $m\geq n$, and let $f:[m]\twoheadrightarrow[n]$.
Then
\[
 \Phi_f:\Mat_n(A)\longrightarrow \Mat_m(A),\qquad
 \Phi_f(X)_{ij}=X_{f(i),f(j)}
\]
is an injective homomorphism. In particular,
$\V(\Mat_n(A))\subseteq\V(\Mat_m(A))$.
\end{lem}
\begin{proof}
It is immediate that $\Phi_f$ preserves addition. Let $X,Y\in \Mat_n(A)$.
Since $f$ is surjective and addition is idempotent, we have
\[
 (\Phi_f(X)\Phi_f(Y))_{ij}
 =\bigoplus_{k=1}^m X_{f(i),f(k)}Y_{f(k),f(j)}
 =\bigoplus_{t=1}^n X_{f(i),t}Y_{t,f(j)}
 =\Phi_f(XY)_{ij}.
\]
Thus $\Phi_f$ also preserves multiplication. For each $a\in[n]$, choose
$i_a\in f^{-1}(a)$. Then $X_{ab}=\Phi_f(X)_{i_a,i_b}$ for all $a,b$,
which shows that $\Phi_f$ is injective. Therefore $\Phi_f$ is a semiring
embedding. This construction is also recorded in
\cite[Theorem 2.2]{S7note}.
\end{proof}

For upper triangular matrices, the standard zero yields the embedding
\begin{equation}\label{eq:utembedding}
 \UT_n(\overline S)\longrightarrow\UT_{n+1}(\overline S),
 \qquad X\longmapsto
 \begin{pmatrix}X&-\infty\\-\infty&-\infty\end{pmatrix}.
\end{equation}
Here an off-diagonal block labelled $-\infty$ has every entry $-\infty$.
The embedding need not preserve the multiplicative identity; no identity
constant is in the signature.

For an even integer $d=2n$ and an ordered tuple $\mathbf z=(z_1,\ldots,z_d)$,
put
\begin{equation}\label{eq:evenodd}
 E_d(\mathbf z)=\bigoplus_{\sgn(\pi)=1}z_{\pi(1)}\cdots z_{\pi(d)},
 \qquad
 O_d(\mathbf z)=\bigoplus_{\sgn(\pi)=-1}z_{\pi(1)}\cdots z_{\pi(d)}.
\end{equation}
The sums range over permutations in the symmetric group on $[d]$.

\begin{lem}\label{lem:AL}
If $A$ is a commutative ai-semiring, then
\begin{equation}\label{eq:AL}
 \Mat_n(A)\models E_{2n}(z_1,\ldots,z_{2n})
 \approx O_{2n}(z_1,\ldots,z_{2n}).
\end{equation}
Consequently this identity holds in both $\Fvar_n$ and $\Uvar_n$.
\end{lem}
\begin{proof}
Let $Z_1,\ldots,Z_{2n}$ be generic $n\times n$ matrices over the
commutative polynomial ring over $\Z$ in their entries. 
Each
entry on either side is a polynomial with nonnegative integer
coefficients in commuting indeterminates. Equality over $\Z$ therefore
means equality coefficient by coefficient, before any evaluation.
Evaluating these positive polynomials in a commutative ai-semiring gives
\eqref{eq:AL}; a positive integer coefficient is interpreted as repeated
addition. No subtraction, additive inverse, or named constant is used
in the resulting semiring identity. The upper triangular assertion
follows by taking a subalgebra of $\Mat_n(\overline S)$.
\end{proof}

\begin{remark}\label{rem:UTparity}
For the upper triangular assertion, one can avoid the ring theorem.
A path for a word of length $2n$ in $n$ nondecreasing states has at most
$n-1$ strict advances, hence at least $n+1$ loops. Some state carries
at least two loop positions. Interchanging the two variable occurrences
at these positions changes the parity of their order but leaves the
path weight unchanged, since the two scalar loop factors commute.
Thus every summand represented on one side of \eqref{eq:AL} is also
represented on the other. The argument applies in both directions.
\end{remark}

\begin{lem}\label{lem:strict}
The identity \eqref{eq:AL} fails in $\Mat_{n+1}(S)$ and in
$\UT_{n+1}(\overline S)$.
\end{lem}
\begin{proof}
Consider the matrix indices $1,\ldots,n+1$. For $1\leq i\leq n$,
assign to $z_{2i-1}$ a matrix whose only entry of weight $1$ is in
position $(i,i)$, and assign to $z_{2i}$ a matrix whose only entry of
weight $1$ is in position $(i,i+1)$.
In the full-matrix case set all other entries to the numerical value $0$.
In the upper triangular case set all other upper triangular entries to
$0$ and all entries below the diagonal to $-\infty$.

At coordinate $(1,n+1)$, the word $z_1\cdots z_{2n}$ has weight $2n$.
Any permutation word with this weight must use only the designated
weight-$1$ entries. Since each letter is used once, it must traverse
all the loops and advances. The only possible order is
\[
 (1,1),(1,2),(2,2),(2,3),\ldots,(n,n),(n,n+1).
\]
This is the identity permutation, which is even. Every odd permutation
word has weight at most $2n-1$ at this coordinate. Therefore the two
sides of \eqref{eq:AL} differ.
\end{proof}

It follows from Lemmas~\ref{lem:embedding}--\ref{lem:strict} and
\eqref{eq:utembedding} that both matrix-variety chains are strictly
ascending. To determine the cardinality of the intervals between their
members, we shall construct an independent family of finite algebras
and identities.

\section{Finite syntactic semirings in matrix varieties}
\label{sec:syntactic}

In this section, we construct finite ai-semirings from finite languages
and show that they belong to the required matrix varieties. To this
end, we first fix the language notation.

For a finite alphabet $\Alphabet$, let $\Alphabet^+$ be the set of nonempty words,
and let $\Alphabet^*$ also include the empty word $\eps$. The empty word will
be used only as a context, not as an algebra element or a term constant.
A \emph{multilinear word} means a word in which no letter occurs twice.
The free ai-semiring on $\Alphabet$ is the set of nonempty finite languages
$P\subseteq\Alphabet^+$, with union as addition and language concatenation
as multiplication.

\begin{defn}\label{def:syntactic}
Let $L\subseteq\Alphabet^+$ be finite and closed under nonempty contiguous
factors. For nonempty finite languages $P,Q$, define
\begin{equation}\label{eq:syntactic}
 P\equiv_L Q
 \quad\Longleftrightarrow\quad
 \bigl(aPb\subseteq L\ \Longleftrightarrow\ aQb\subseteq L\bigr)
 \quad\text{for all }a,b\in\Alphabet^*.
\end{equation}
Let $\Syn(L)$ be the quotient by this relation.
\end{defn}

In \eqref{eq:syntactic}, the condition $aPb\subseteq L$ requires every
word of $aPb$ to belong to $L$. Thus it is a universal acceptance
condition, rather than a condition asserting that at least one word
is accepted.

\begin{lem}\label{lem:syntactic}
The relation $\equiv_L$ is a semiring congruence. The algebra $\Syn(L)$
is a finite ai-semiring, generated by the classes of letters, and
\[
 |\Syn(L)|\leq 2^{|L|}.
\]
All languages containing a word outside $L$ belong to one class
$\top$, which absorbs both operations. If every word of $L$ has length
at most $d$, every product of $d+1$ elements of $\Syn(L)$ equals $\top$.
Moreover,
\begin{equation}\label{eq:orderprofiles}
 [P]\leq[Q]
 \quad\Longleftrightarrow\quad
 (aQb\subseteq L\Longrightarrow aPb\subseteq L)
 \quad(a,b\in\Alphabet^*).
\end{equation}
\end{lem}
\begin{proof}
Let $P\equiv_L Q$, and let $R$ be a nonempty finite language.
For any contexts $a,b$, the inclusion $a(P\cup R)b\subseteq L$
holds if and only if both $aPb\subseteq L$ and $aRb\subseteq L$.
It follows that $P\cup R\equiv_L Q\cup R$.
Similarly, $a(PR)b\subseteq L$ holds if and only if $aPwb\subseteq L$
for every $w\in R$. Applying $P\equiv_L Q$ to the contexts
$(a,wb)$ gives $PR\equiv_L QR$. The verification for left
multiplication is analogous. Hence $\equiv_L$ is a semiring congruence.

If $w\notin L$, no word $awb$ belongs to $L$, by factorial closure.
Consequently every language containing $w$ has an empty acceptance
profile, independently of which such $w$ it contains. These languages
form the class $\top$. Every other class has a representative that is a
nonempty subset of $L$, proving the bound
$(2^{|L|}-1)+1$. The assertions about generation and absorption follow
from the free-language construction and factorial closure. A product of
$d+1$ nonempty positive languages contains only words of length at least
$d+1$, so its class is $\top$. Finally,
$[P]\leq[Q]$ means $P\cup Q\equiv_L Q$, which is exactly
\eqref{eq:orderprofiles}.
\end{proof}

We now establish the matrix membership result that will be used in
the construction of the subvariety intervals.

\begin{thm}\label{thm:recognition}
Fix $n\geq1$, set $d=2n$, and let $\Alphabet$ be a finite alphabet with at
least $d$ letters. Let $\mathcal R$ be any set of multilinear words of
length $d$, and define
\begin{equation}\label{eq:languagegeneral}
 L=L_{d,\mathcal R}
 =\{w\in\Alphabet^+:w\text{ is multilinear and }|w|<d\}\ \cup\ \mathcal R.
\end{equation}
Then
\begin{equation}\label{eq:recognition}
 \Syn(L)\in\Fvar_{n+1}\cap\Uvar_{n+1}.
\end{equation}
In fact, $\Syn(L)$ is a homomorphic image of a subalgebra of a finite
direct power of either matrix algebra occurring in \eqref{eq:recognition}.
\end{thm}
\begin{proof}
It is clear from the definition that $L$ is closed under nonempty
contiguous factors. Let $v\notin L$. We first construct a matrix
assignment for which the value of $v$ is strictly larger than the value
of every word of $L$ at a fixed coordinate. The construction applies
to both full and upper triangular matrices of size $n+1$. We consider
the following cases.

\smallskip\noindent\textit{Case 1.} The word $v$ repeats a letter $a$.
Assign every entry of the matrix for $a$ the numerical value $1$ and
every entry for any other letter the value $0$. In the triangular
version, this prescription applies only on and above the diagonal;
entries below it are $-\infty$. At every upper triangular coordinate,
a word has value equal to its number of occurrences of $a$.
Every word in $L$ has value at most $1$, whereas $v$ has value at least $2$.

\smallskip\noindent\textit{Case 2.} We have $|v|>d$.
Assign every letter the constant weight-$1$ matrix, or its upper
triangular version. At coordinate $(1,n+1)$ a word has value its length.
The values of words in $L$ are at most $d$, and the value of $v$ is larger.

\smallskip\noindent\textit{Case 3.} The word $v$ is multilinear and $|v|\leq d$.
Since $v\notin L$, necessarily $|v|=d$ and $v\notin\mathcal R$.
Write $v=v_1\cdots v_{2n}$. Assign the letters $v_{2i-1}$ and $v_{2i}$
the loop and advance matrices from Lemma~\ref{lem:strict}; assign any
other letter a weight-$0$ matrix. The value of $v$ at $(1,n+1)$ is $d$.
A word of $L$ could attain $d$ only if it had length $d$ and used only
weight-$1$ entries. Its multilinearity would then force it to use all
$2n$ designated letters exactly once. The unique possible path order
would force the word to be $v$, a contradiction. Every word in $L$
therefore has value at most $d-1$ at this coordinate.

The preceding cases give a finite set $\mathscr E$ of assignments:
one counting assignment for each letter, one length assignment, and
one loop--advance assignment for each forbidden multilinear word of
length $d$. Let $T$ denote either target matrix algebra and let
\[
 \Psi:\Pal_f^+(\Alphabet^+)\longrightarrow T^{\mathscr E}
\]
be the simultaneous evaluation homomorphism on the free ai-semiring.
Here $\Pal_f^+$ denotes nonempty finite subsets.

We claim that $\ker\Psi\subseteq\equiv_L$. Suppose
$\Psi(P)=\Psi(Q)$ but, for some contexts $a,b$,
$aPb\subseteq L$ and $aQb\not\subseteq L$. Choose
$v\in aQb\setminus L$. One of the assignments in $\mathscr E$ makes
$v$ strictly larger than every word of $L$ at the designated
coordinate. Therefore it makes the matrix value of $aQb$ strictly
larger there than the value of $aPb$, contradicting
$\Psi(aPb)=\Psi(aQb)$. The other direction is symmetric.
Empty contexts cause no difficulty: no matrix identity element needs
to be inserted.

Consequently, the map $\Psi(P)\mapsto[P]$ is a well-defined semiring
epimorphism from $\operatorname{im}\Psi$ onto $\Syn(L)$. Since
$\operatorname{im}\Psi$ is a subsemiring of $T^{\mathscr E}$ and
$\mathscr E$ is finite, the required assertions follow.
\end{proof}

\begin{remark}\label{rem:notflat}
In general, the ai-semiring $\Syn(L)$ is not flat. The preceding
membership proof uses its specified addition and does not apply to a
flat factor semiring in its place.
For the languages in \eqref{eq:languagegeneral}, $L$ is even closed under
nonempty scattered subwords. Consequently $\Syn(L)$ satisfies
$x\oplus y\leq xy$: acceptance of every word of $aPQb$ implies
acceptance of $aPb$ and $aQb$ by deleting the other factor.
This is consistent with the identity $x\oplus y\leq xy$ in every
$\Mat_k(S)$. Indeed,
$(AB)_{ij}\geq A_{ij}+B_{jj}\geq A_{ij}$ and
$(AB)_{ij}\geq A_{ii}+B_{ij}\geq B_{ij}$.
\end{remark}

\section{Prime-cycle languages and independent identities}
\label{sec:cycles}

We next apply Theorem~\ref{thm:recognition} to a family of languages
associated with directed cycles. The corresponding finite semirings
will be distinguished by an independent family of identities.

Fix $n\geq1$ and write $d=2n$. For an odd integer $p\geq3$, let
\[
 \Alphabet_{n,p}=\{a_1,\ldots,a_{d-2}\}\ \dot\cup\
 \{v_0,\ldots,v_{p-1}\}.
\]
The first set is empty when $n=1$. Throughout this construction,
indices on the $v_i$ are read modulo $p$. For each directed edge $v_i\to v_{i+1}$ of the cycle, consider the
ordered tuple
\[
 (a_1,\ldots,a_{d-2},v_i,v_{i+1}).
\]
Let $\mathcal R_{n,p}$ contain all even permutations of all these tuples,
viewed as words. Define
\begin{equation}\label{eq:Anp}
 L_{n,p}=L_{2n,\mathcal R_{n,p}},\qquad
 \Alg_{n,p}=\Syn(L_{n,p}).
\end{equation}
Each longest-word support contains all the $d-2$ anchors and exactly two
adjacent cycle vertices. On each such support, exactly one of the two
permutation parities is accepted. Since $p\geq3$, no reverse edge is a
second directed edge with the same support. In particular,
\[
 |\mathcal R_{n,p}|=p\frac{(2n)!}{2}.
\]
By Theorem~\ref{thm:recognition},
\begin{equation}\label{eq:Anpmembership}
 \Alg_{n,p}\in\Fvar_{n+1}\cap\Uvar_{n+1}.
\end{equation}

Use pairwise distinct formal variables
\[
 z_1,\ldots,z_{d-2},x_0,\ldots,x_{p-1},
\]
with the $x$-indices cyclic modulo $p$. Set
\begin{align}
 P_{n,p}
 &=\bigoplus_{i=0}^{p-1}
 E_d(z_1,\ldots,z_{d-2},x_i,x_{i+1}),\label{eq:Pnp}\\
 q_{n,p}&=z_1\cdots z_{d-2}x_1x_0,\label{eq:qnp}\\
 \sigma_{n,p}&:\quad q_{n,p}\leq P_{n,p}.\label{eq:sigmanp}
\end{align}
When $n=1$, \eqref{eq:qnp} is simply $x_1x_0$; no empty-word constant
is present.

\begin{lem}\label{lem:lower}
Both $\Fvar_n$ and $\Uvar_n$ satisfy $\sigma_{n,p}$ for every odd $p\geq3$.
\end{lem}
\begin{proof}
The word $q_{n,p}$ is an odd permutation of
$(z_1,\ldots,z_{d-2},x_0,x_1)$. Therefore
\[
 q_{n,p}\leq O_d(z_1,\ldots,z_{d-2},x_0,x_1)
 \approx E_d(z_1,\ldots,z_{d-2},x_0,x_1)
 \leq P_{n,p},
\]
using Lemma~\ref{lem:AL} for the middle equality.
\end{proof}

\begin{lem}\label{lem:cyclemaps}
Let $p,q\geq3$ be odd. There is a map from the variables of $P_{n,p}$
to $\Alphabet_{n,q}$ sending every word of $P_{n,p}$ into $\mathcal R_{n,q}$
if and only if $q\mid p$.
Every such map sends $q_{n,p}$ outside $\mathcal R_{n,q}$.
\end{lem}
\begin{proof}
Suppose that a map $f$ with the stated property exists. Put
\[
 C=\{a_1,\ldots,a_{d-2}\},\qquad
 A=\{f(z_1),\ldots,f(z_{d-2})\},\qquad y_i=f(x_i).
\]
Every accepted word is multilinear. Thus the anchor images are distinct,
$y_i\notin A$, and $y_i\ne y_{i+1}$. Moreover, for every $i$,
\begin{equation}\label{eq:supportanchor}
 A\cup\{y_i,y_{i+1}\}=C\cup\{v_j,v_{j+1}\}
 \quad\text{for some }j.
\end{equation}
We first show that $A=C$. This also covers the case in which the
anchor set is empty.

Put $r=|C\setminus A|=|A\setminus C|$. Equation
\eqref{eq:supportanchor} gives $r\leq2$.
If $r=1$, let $c$ be the missing anchor. Every pair
$\{y_i,y_{i+1}\}$ must contain $c$ and one element different from $c$.
Thus the sequence alternates between $c$ and an element different from
$c$, which is impossible because $p$ is odd. If $r=2$, every such pair
consists of the two missing anchors. This again forces alternation
around an odd cycle, a contradiction. Hence $r=0$ and $A=C$.

The ordered anchor images are now a fixed permutation of the anchors.
Let its sign be $s\in\{1,-1\}$. If there are no anchors, take $s=1$.
Each $y_i$ is a cycle vertex, say $y_i=v_{b_i}$.
Acceptance of the ordered word
$f(z_1)\cdots f(z_{d-2})y_i y_{i+1}$ means precisely that the directed
edge is traversed forwards if $s=1$, and backwards if $s=-1$.
Indeed, reversing the last two positions changes permutation parity.
Consequently
\[
 b_{i+1}=b_i+s\pmod q\quad\text{for every }i.
\]
Returning to the starting vertex after $p$ steps forces $q\mid p$.

Conversely, when $q\mid p$, map $z_j\mapsto a_j$ and
$x_i\mapsto v_{i\bmod q}$. Every even edge permutation is accepted.
Finally, under any map in the statement, the variables in the first
edge tuple have distinct images. Its even permutations are accepted,
whereas the word $q_{n,p}$ reverses the last two positions and is odd
on the same support. Exactly one parity is accepted on this support,
so $f(q_{n,p})\notin\mathcal R_{n,q}$.
\end{proof}

\begin{thm}\label{thm:independent}
For $n\geq1$ and odd $p,q\geq3$,
\begin{equation}\label{eq:truthdiv}
 \Alg_{n,q}\models\sigma_{n,p}\quad\Longleftrightarrow\quad q\nmid p.
\end{equation}
In particular, for odd primes $p,q$,
\begin{equation}\label{eq:truthprime}
 \Alg_{n,q}\models\sigma_{n,p}\quad\Longleftrightarrow\quad p\ne q.
\end{equation}
\end{thm}
\begin{proof}
Suppose first that $q_{n,p}\leq P_{n,p}$ fails in $\Syn(L_{n,q})$.
Choose an assignment witnessing this failure, and represent the value
of each variable by a nonempty finite language of positive words.
Let $\phi$ denote the resulting substitution. By \eqref{eq:orderprofiles}, there exist contexts $a,b$
with
\begin{equation}\label{eq:acceptancefailure}
 a\phi(P_{n,p})b\subseteq L_{n,q},\qquad
 a\phi(q_{n,p})b\not\subseteq L_{n,q}.
\end{equation}
Every word of $P_{n,p}$ has length $d$, and every substitution word has
positive length. Since the target language has maximum length $d$,
the first condition forces $a=b=\eps$.
Every variable occurs in $P_{n,p}$. If one of its substitution words
had length at least two, a substituted word of $P_{n,p}$ would have
length greater than $d$. Thus every substitution language consists of
single letters.

Choose one letter for each variable, making the choice for the
variables of $q_{n,p}$ agree with a word witnessing the second condition
in \eqref{eq:acceptancefailure}; extend the choice arbitrarily to the
remaining variables. This gives a map $f$ sending every word of
$P_{n,p}$ into $\mathcal R_{n,q}$. Lemma~\ref{lem:cyclemaps} yields
$q\mid p$. Consequently, if $q\nmid p$, the identity is satisfied
under every assignment in the algebra.

If $q\mid p$, use the map in the converse part of
Lemma~\ref{lem:cyclemaps}, and assign the corresponding letter classes.
The image of $P_{n,p}$ is a subset of $L_{n,q}$, while the image of
$q_{n,p}$ is not. The empty context distinguishes the two sides, so
the identity fails. This proves \eqref{eq:truthdiv} and its prime
specialization.
\end{proof}

\begin{example}
For $n=1$ there are no anchors. The longest accepted words in
$L_{1,p}$ are $v_i v_{i+1}$, and all single letters are accepted.
The identity is
\[
 x_1x_0\leq x_0x_1\oplus x_1x_2\oplus\cdots\oplus x_{p-1}x_0.
\]
It holds in the scalar max-plus algebras by commutativity, and its
failure in $\Alg_{1,q}$ is equivalent to a closed directed walk of
length $p$ in a directed $q$-cycle.
For higher dimensions, the $2n-2$ anchors and the even permutations
ensure that the witnesses lie in the appropriate adjacent interval.
The construction for the first interval alone does not give this
conclusion.
\end{example}

\section{The intervals between matrix varieties}\label{sec:continuum}

We are now in a position to prove that every interval between distinct
matrix dimensions has cardinality $2^{\aleph_0}$. We first construct
explicit order embeddings into the adjacent intervals.

\begin{thm}\label{thm:intervals}
For $n\geq1$ and $J\subseteq\PP$, define
\begin{align}
 \mathcal X^F_{n,J}
 &=\Fvar_n\vee\bigvee_{p\in J}\V(\Alg_{n,p}),\label{eq:joinF}\\
 \mathcal X^U_{n,J}
 &=\Uvar_n\vee\bigvee_{p\in J}\V(\Alg_{n,p}).\label{eq:joinU}
\end{align}
For the empty set these are, respectively, $\Fvar_n$ and $\Uvar_n$.
The two maps $J\mapsto\mathcal X^F_{n,J}$ and
$J\mapsto\mathcal X^U_{n,J}$ are join-preserving order embeddings into
$[\Fvar_n,\Fvar_{n+1}]$ and $[\Uvar_n,\Uvar_{n+1}]$.
\end{thm}
\begin{proof}
By \eqref{eq:Anpmembership} and the dimension embeddings, the varieties
in \eqref{eq:joinF} and \eqref{eq:joinU} belong to their respective
intervals. Lemma~\ref{lem:lower} and
Theorem~\ref{thm:independent} give, for $\bullet\in\{F,U\}$,
\begin{equation}\label{eq:jointruth}
 \mathcal X^\bullet_{n,J}\models\sigma_{n,p}
 \quad\Longleftrightarrow\quad p\notin J.
\end{equation}
Indeed, an identity holds in a join exactly when it holds in all the
joining varieties.

If $J\subseteq K$, it follows immediately from the definitions that
$\mathcal X^\bullet_{n,J}\subseteq\mathcal X^\bullet_{n,K}$.
Conversely, suppose that $J\nsubseteq K$, and choose $p\in J\setminus K$.
By \eqref{eq:jointruth}, the variety $\mathcal X^\bullet_{n,K}$ satisfies
$\sigma_{n,p}$, whereas $\mathcal X^\bullet_{n,J}$ does not. Hence
$\mathcal X^\bullet_{n,J}\nsubseteq\mathcal X^\bullet_{n,K}$.
Thus both maps preserve and reflect inclusion. Their defining formulas
also show that arbitrary unions of index sets correspond to joins in
the respective intervals, whose least elements are the stated lower
endpoints. This completes the proof.
\end{proof}

\begin{cor}\label{cor:cardinality}
For all $1\leq n<m$,
\[
 |[\Fvar_n,\Fvar_m]|=|[\Uvar_n,\Uvar_m]|=2^{\aleph_0}.
\]
Each of these intervals contains a chain and an antichain of that
cardinality.
\end{cor}
\begin{proof}
The countably infinite set of odd primes has $2^{\aleph_0}$ subsets,
so Theorem~\ref{thm:intervals} supplies this lower bound already in the
adjacent interval. There are only countably many identities in the
fixed finite signature over countably many variables, hence at most
$2^{\aleph_0}$ equational theories. This is the matching upper bound.
An interval with $m>n+1$ contains the adjacent one.

For completeness, the power set of a countably infinite set contains
a continuum-sized chain: identify the set with $\Q$ and use the
sets $\{q\in\Q:q<t\}$ for $t\in\R$. It also contains a
continuum-sized antichain: partition the set into pairs and choose
exactly one member of each pair. Transport these families through
either order embedding.
\end{proof}

We do not claim that these order embeddings preserve meets in the
ambient subvariety lattice. Nor do we claim that every variety in either
interval is one of the displayed joins. The independent identities
suffice for the asserted cardinalities.

\section{Finite commutative semirings from divisor languages}
\label{sec:algebra}

In this section, we construct the finite commutative ai-semirings needed
for the nonfinite basis proofs. Their operations are induced by a
syntactic congruence on finite sets of commutative monomials. We first
establish the congruence construction, then prove four auxiliary lemmas,
and finally verify the bounded-variable property under arbitrary
assignments.

\subsection{A syntactic congruence for commutative monomials}

Let $\mathsf M_m$ be the free commutative monoid on the auxiliary letters
$X_1,\ldots,X_m$. Its elements are monomials
$u=X_1^{a_1}\cdots X_m^{a_m}$, with $a_j\in\N$, including the monoid
identity $1$. Write $u\mid v$ for ordinary monomial divisibility and
$\deg u=\sum_j a_j$. Let $\mathscr P_m$ be the ai-semiring of nonempty
finite subsets of $\mathsf M_m$, with
\[
 P\oplus Q=P\cup Q,\qquad PQ=\{uv:u\in P,\ v\in Q\}.
\]
It has commutative multiplication and an actual identity $\{1\}$.
This auxiliary identity is an element, not a new symbol in the
signature of the identities under consideration.

Let $\Delta\subseteq\mathsf M_m$ be finite, nonempty, and closed under
divisors. For $P\in\mathscr P_m$, define its \emph{residual set} by
\begin{equation}\label{eq:residual}
 \rho_\Delta(P)=\{c\in\mathsf M_m:cP\subseteq\Delta\}.
\end{equation}
The acceptance condition is universal over the monomials of $P$.
Define
\begin{equation}\label{eq:moncongruence}
 P\sim_\Delta Q\quad\Longleftrightarrow\quad
 \rho_\Delta(P)=\rho_\Delta(Q),\qquad
 C(\Delta)=\mathscr P_m/{\sim_\Delta}.
\end{equation}

\begin{lem}\label{lem:moncongruence}
The relation $\sim_\Delta$ is a semiring congruence. The algebra
$C(\Delta)$ is a finite commutative ai-semiring with multiplicative
identity $[\{1\}]$, and
\[
 |C(\Delta)|\leq 2^{|\Delta|}.
\]
Every polynomial containing a monomial outside $\Delta$ has the same
class $\top$, which absorbs both operations. For arbitrary $P,Q$,
\begin{equation}\label{eq:monorder}
 [P]\leq[Q]\quad\Longleftrightarrow\quad
 \rho_\Delta(Q)\subseteq\rho_\Delta(P).
\end{equation}
\end{lem}
\begin{proof}
Since
$\rho_\Delta(P\cup Q)=\rho_\Delta(P)\cap\rho_\Delta(Q)$, addition
preserves equivalence. For any nonempty finite monomial set $R$,
\[
 cPR\subseteq\Delta
 \quad\Longleftrightarrow\quad
 crP\subseteq\Delta\text{ for every }r\in R.
\]
Thus $P\sim_\Delta Q$ implies $PR\sim_\Delta QR$; commutativity
gives the other side. The quotient therefore inherits every semiring
law and its multiplicative identity from $\mathscr P_m$.

If $u\notin\Delta$, divisor closure implies $cu\notin\Delta$ for
all $c\in\mathsf M_m$. Consequently a polynomial containing $u$ has
empty residual set. Such polynomials form one class $\top$, absorbing
addition and multiplication. Every other class has a nonempty
representative contained in $\Delta$. There are at most
$2^{|\Delta|}-1$ such representatives and one additional class.
Finally, $[P]\leq[Q]$ means $[P\cup Q]=[Q]$, which gives
\eqref{eq:monorder}. Notice also that every residual set is contained
in $\Delta$: if $cP\subseteq\Delta$, then $c$ divides any monomial
in the nonempty set $cP$.
\end{proof}

\subsection{The cyclic divisor languages}

Fix integers $k\geq1$ and $m\geq3$, and use cyclic indices modulo $m$.
In $\mathsf M_m$, put
\begin{equation}\label{eq:cyclicmonomials}
 \mu=(X_1\cdots X_m)^k,\qquad
 \nu_i=X_i^{k+1}X_{i^+}^{k-1}
       \prod_{j\notin\{i,i^+\}}X_j^k\quad(1\leq i\leq m).
\end{equation}
Thus $\nu_i=\mu X_i/X_{i^+}$ in the group of Laurent monomials;
\eqref{eq:cyclicmonomials} shows that all exponents in the monoid
itself are nonnegative. Define
\begin{equation}\label{eq:deltaalgebra}
 \Delta_{k,m}=\bigcup_{i=1}^m\{u:u\mid\nu_i\},\qquad
 \Gamma_{k,m}=\Delta_{k,m}\cup\{\mu\},\qquad
 C_{k,m}=C(\Delta_{k,m}).
\end{equation}
The definition uses only ordinary monomial divisibility. In the
following proofs, we suppress the subscripts on $\Delta$ and $\Gamma$
and put $D=km$.

The set $\Delta$ is divisor-closed, but $\mu\notin\Delta$, since
$\nu_i$ has exponent $k-1$ at $X_{i^+}$. Every proper divisor of
$\mu$ belongs to $\Delta$: if its exponent at $X_j$ is less than
$k$, it divides $\nu_{j^-}$. Consequently
\begin{equation}\label{eq:deltacardinality}
 |\Delta_{k,m}|=(k+1)^m-1+mk(k+1)^{m-2}.
\end{equation}
Indeed, the first term counts the proper divisors of $\mu$.
A remaining divisor has exactly one exponent equal to $k+1$;
choose its position $i$, its successor exponent in $\{0,\ldots,k-1\}$,
and the other exponents in $\{0,\ldots,k\}$.
The only monomials of degree $D$ in $\Gamma$ are
\begin{equation}\label{eq:degreed}
 \mu,\nu_1,\ldots,\nu_m.
\end{equation}

For $\mathbf a=(a_1,\ldots,a_m)\in\N^m$ and
$u=X_1^{b_1}\cdots X_m^{b_m}$, set
\[
 \operatorname{wt}_{\mathbf a}(u)=\sum_j b_ja_j,
 \qquad
 \operatorname{wt}_{\mathbf a}(P)=
 \max_{u\in P}\operatorname{wt}_{\mathbf a}(u).
\]
These weights are evaluations in $S$. In particular, the auxiliary
monomial $1$ has weight $0$, which is the multiplicative identity of $S$.
Write $Q_{k,m}=\{\nu_1,\ldots,\nu_m\}$ when it is viewed as a
polynomial in $\mathscr P_m$.

\subsection{Auxiliary lemmas}

The following four lemmas provide the algebraic properties used in
the bounded-variable argument.

\begin{lem}
\label{lem:monomialbound}
For any monomial $u\in\mathsf M_m$,
\begin{equation}\label{eq:monomialbound}
 \operatorname{wt}_{\mathbf a}(u)
 \leq\operatorname{wt}_{\mathbf a}(Q_{k,m})
 \quad\text{for all }\mathbf a\in\N^m
 \quad\Longleftrightarrow\quad u\in\Gamma_{k,m}.
\end{equation}
\end{lem}
\begin{proof}
If $u\mid\nu_i$, the inequality follows from nonnegativity of the
weights. For $u=\mu$, note that
\[
 \operatorname{wt}_{\mathbf a}(\nu_i)
 =\operatorname{wt}_{\mathbf a}(\mu)+a_i-a_{i^+}.
\]
At least one cyclic difference is nonnegative, proving the inequality.

Conversely, write $u=\prod_jX_j^{b_j}$ and assume the inequality.
Assign $a_j=1$ and every other weight zero. This gives $b_j\leq k+1$
for each $j$. If every $b_j\leq k$, then $u\mid\mu$, and the
remarks preceding \eqref{eq:deltacardinality} give $u\in\Gamma$.
Otherwise fix $i$ with $b_i=k+1$. For each $j\ne i$, assign
\[
 a_i=3,\qquad a_j=1,\qquad a_h=0\quad(h\notin\{i,j\}).
\]
Among the $\nu_h$, the monomial $\nu_i$ has strictly largest weight:
it gains at least $3$ in coordinate $i$ relative to any other
$\nu_h$, while a loss in coordinate $j$ is at most $2$.
The assumed inequality therefore gives
\[
 3(k+1)+b_j\leq 3(k+1)+\operatorname{exp}_{X_j}(\nu_i).
\]
Thus $b_j\leq\operatorname{exp}_{X_j}(\nu_i)$ for every $j\ne i$,
and equality already holds at $i$. Hence $u\mid\nu_i$.
\end{proof}

\begin{lem}\label{lem:compulsory}
Let $B\subseteq\Delta_{k,m}$ be nonempty. If
\[
 \operatorname{wt}_{\mathbf a}(\mu)
 \leq\operatorname{wt}_{\mathbf a}(B)
 \quad\text{for all }\mathbf a\in\N^m,
\]
then $\nu_i\in B$ for every $i$.
\end{lem}
\begin{proof}
Fix $i$ and suppose $\nu_i\notin B$. Choose integers $b_j$ by
\[
 b_{i+1}=0,\ b_{i+2}=1,\ \ldots,\ b_i=m-1
\]
with cyclic subscripts. Put $T=D(m-1)+1$ and $a_j=T+b_j$.
For $j\ne i$ one has $b_j-b_{j^+}=-1$, so
\[
 \operatorname{wt}_{\mathbf a}(\nu_j)
 =\operatorname{wt}_{\mathbf a}(\mu)-1.
\]
For any monomial $u$ of degree at most $D-1$,
\[
 \operatorname{wt}_{\mathbf a}(u)
 \leq(D-1)(T+m-1)<DT
 \leq\operatorname{wt}_{\mathbf a}(\mu).
\]
By \eqref{eq:degreed}, these two estimates cover every monomial of
$B$: neither $\mu$ nor $\nu_i$ belongs to $B$.
They contradict the assumed inequality at this explicit integer
assignment. Thus every $\nu_i$ must belong to $B$.
\end{proof}

\begin{lem}\label{lem:productrigidity}
If $u,v\in\Gamma_{k,m}$ and $uv=\mu^2$ as ordinary monomials,
then $u=v=\mu$.
\end{lem}
\begin{proof}
Every element of $\Gamma$ has degree at most $D$, so both factors
have degree $D$. Apply \eqref{eq:degreed}. A factor equal to $\mu$
forces the other to equal $\mu$ by cancellation. The only remaining
possibility would be $\nu_i\nu_j=\mu^2$. Cancelling $\mu^2$ in
the group of Laurent monomials gives
\[
 X_iX_j=X_{i^+}X_{j^+}.
\]
The multiplicity vector of the two indices on the left would be
unchanged by cyclic shift, so all its $m$ entries would be equal.
Their sum is $2$, impossible for integers when $m\geq3$.
\end{proof}

\begin{lem}\label{lem:abelianrank}
Suppose monomials $c,u_1,\ldots,u_\ell\in\mathsf M_m$ have the
property that every $\nu_i$ is of the form
\[
 \nu_i=c\prod_{j=1}^\ell u_j^{a_{ij}},\qquad a_{ij}\in\N.
\]
Then $\ell\geq m$.
\end{lem}
\begin{proof}
The common factor $c$ divides every $\nu_i$, so it divides
$\gcd(\nu_1,\ldots,\nu_m)=(X_1\cdots X_m)^{k-1}$.
Consequently all $w_i=\nu_i/c$ are monomials of the same positive
degree $D-\deg c\geq m$.

We claim that $w_1,\ldots,w_m$ are independent in the free abelian
group of Laurent monomials. Indeed, suppose that
$\prod_iw_i^{b_i}=1$ for integers $b_i$. Since all the $w_i$ have the
same positive degree, comparison of degrees gives $\sum_i b_i=0$.
Cancelling the common factors $\mu$ and $c$, we obtain
$\prod_i(X_i/X_{i^+})^{b_i}=1$. Comparison of the exponents of $X_j$
now gives $b_j=b_{j^-}$ for every $j$. The integers $b_i$ are therefore
equal, and their sum is zero. Hence every $b_i$ is zero, proving the
claim.

All $w_i$ belong to the subgroup generated by $u_1,\ldots,u_\ell$.
A subgroup generated by $\ell$ elements has rank at most $\ell$;
equivalently, tensor with $\Q$ and count vector-space generators.
It follows that $\ell\geq m$, as required.
\end{proof}

\subsection{The bounded-variable property}

We now establish the main result of this section. The assertion concerns
all assignments in the finite algebra, not only assignments to its
generators.

\begin{thm}\label{thm:algebraiclocality}
For every $k\geq1$ and $m\geq3$, the finite commutative ai-semiring
$C_{k,m}$ satisfies
\begin{equation}\label{eq:countermodeltheorem}
 \begin{gathered}
 C_{k,m}\models\Id_{m-1}(S),\qquad
 C_{k,m}\not\models\eta_{k,m},\\[3pt]
 \eta_{k,m}:\quad (x_1\cdots x_m)^k\leq
 \bigoplus_{i=1}^m x_i^{k+1}x_{i^+}^{k-1}
 \prod_{j\notin\{i,i^+\}}x_j^k.
 \end{gathered}
\end{equation}
The products on the right can be put in any fixed order. Zero
exponents mean omission; every monomial has positive length $km$.
The locality assertion includes assignments to $\top$, to the
multiplicative identity, and to arbitrary sums and products.
\end{thm}
\begin{proof}
Let $s\approx t$ be valid in $S$ and involve $\ell\leq m-1$
variables $y_1,\ldots,y_\ell$. Represent an arbitrary assignment in
$C_{k,m}$ by nonempty finite monomial sets $P_1,\ldots,P_\ell$.
After expansion, write $s(P)$ and $t(P)$ for the resulting elements
of $\mathscr P_m$.

To prove that the assigned values are equal, it suffices to show that
$s(P)$ and $t(P)$ have the same residual set. Suppose, for a
contradiction, that they do not. Interchanging $s$ and $t$ if necessary,
there exists a monomial $c\in\mathsf M_m$ such that
\begin{equation}\label{eq:localfailure}
 cs(P)\subseteq\Delta,
 \qquad ct(P)\not\subseteq\Delta.
\end{equation}
Every monomial of $cs(P)$ is bounded by $Q_{k,m}$ under every
nonnegative integer assignment, by Lemma~\ref{lem:monomialbound}.
The scalar identity $s\approx t$, with the given polynomial
substitution and multiplication by $c$, gives the same bound for
$ct(P)$. Applying that lemma separately to its monomials yields
\[
 ct(P)\subseteq\Gamma=\Delta\cup\{\mu\}.
\]
Thus the second condition in \eqref{eq:localfailure} says precisely
that $\mu$ occurs in $ct(P)$.

Choose a monomial of the expanded term $t$ and, for each occurrence
of a variable in that monomial, a monomial from its substitution set,
so that the resulting product with $c$ is $\mu$.
Suppose two occurrences of the same $y_j$ choose $u,v\in P_j$.
Replacing the first choice by the second, and conversely replacing
the second by the first, produces two monomials
\[
 \mu\,v/u\quad\text{and}\quad\mu\,u/v
\]
that both occur in $ct(P)$. They are genuine monomials: the indicated
cancellation removes an actual selected factor. Their product is
$\mu^2$, and both belong to $\Gamma$. By
Lemma~\ref{lem:productrigidity}, both equal $\mu$, and cancellation
gives $u=v$.

Consequently all occurrences of any one variable in the selected
monomial choose a single monomial $u_j\in P_j$. For a variable
absent from that monomial, choose any $u_j\in P_j$. Restrict the
substitution to these single monomials and put
\[
 B=cs(u_1,\ldots,u_\ell).
\]
This is a nonempty subset of $cs(P)$ and hence of $\Delta$.
The selected monomial of $ct(u_1,\ldots,u_\ell)$ remains $\mu$.
Since the restricted substitution still preserves the scalar identity,
\[
 \operatorname{wt}_{\mathbf a}(\mu)
 \leq\operatorname{wt}_{\mathbf a}(B)
 \quad(\mathbf a\in\N^m).
\]
Lemma~\ref{lem:compulsory} forces every $\nu_i$ to occur in $B$.
Each has the form $c\prod_{j=1}^\ell u_j^{a_{ij}}$, because $B$
comes from the term $s$ in $\ell$ variables. Lemma~\ref{lem:abelianrank}
now gives $\ell\geq m$, contradicting $\ell\leq m-1$.
This contradiction excludes \eqref{eq:localfailure}. Hence $s(P)$ and
$t(P)$ have equal residual sets. Since the assignment was arbitrary,
$C_{k,m}$ satisfies $s\approx t$.

It remains to show that $C_{k,m}$ does not satisfy $\eta_{k,m}$.
Assign to $x_j$ the class $[\{X_j\}]$.
The left side of $\eta_{k,m}$ becomes $[\{\mu\}]=\top$, since
$\mu\notin\Delta$. Each $\nu_i$ has degree $D$ and belongs to
$\Delta$, whose monomials all have degree at most $D$.
Therefore
\begin{equation}\label{eq:residualseparation}
 \rho_\Delta(\{\mu\})=\varnothing,\qquad
 \rho_\Delta(\{\nu_i\})=\{1\}\quad(1\leq i\leq m).
\end{equation}
All the right-hand monomials consequently have the same class
$\omega\ne\top$, whose residual set is $\{1\}$.
However, $\top\oplus\omega=\top\ne\omega$. Thus $\eta_{k,m}$ fails
under this assignment, completing the proof.
\end{proof}

\begin{cor}\label{cor:algebraicobstruction}
For any $r\geq1$ and $k\geq1$, put $m=\max\{3,r+1\}$.
The finite algebra $C_{k,m}$ satisfies every $r$-variable identity
of $S$ and fails $\eta_{k,m}$. Consequently those identities do not
imply $\eta_{k,m}$, even with all elementary commutative ai-semiring
laws available as background axioms.
\end{cor}
\begin{proof}
Apply Theorem~\ref{thm:algebraiclocality}, noting that $r\leq m-1$, and use
soundness of equational logic. The countermodel is itself a
commutative ai-semiring.
\end{proof}

\begin{remark}
The construction above does not use their locality theorem or a
geometric characterization of scalar identities. For every $k$, the
semiring laws follow first from the congruence in
\eqref{eq:moncongruence}; the bounded-variable property is then proved
separately. Deleting a monomial from an arbitrary closure construction
would not, without further proof, preserve these laws.
\end{remark}

\begin{example}
For $k=1,m=3$,
\[
 \mu=X_1X_2X_3,\quad
 \nu_1=X_1^2X_3,\quad
 \nu_2=X_1X_2^2,\quad
 \nu_3=X_2X_3^2.
\]
The divisor language $\Delta_{1,3}$ has $13$ monomials, including
$1$, but not $X_1X_2X_3$. 
semiring zero: it absorbs addition as well as multiplication.
The multiplicative identity $[\{1\}]$ is a different element.
\end{example}

\section{The finite basis problem for upper triangular matrices}
\label{sec:triangularnfb}

In this section, we apply Theorem~\ref{thm:algebraiclocality} to
$\UT_n(\overline S)$. We first construct identities that hold in this
matrix semiring and fail in the finite countermodels.

For $m\geq3$, use the words $p_m$ and $q_{m,i}$ from
\eqref{eq:scalarwords}. Recall that the products defining $q_{m,i}$
are taken in increasing variable order. For $n\geq1$, define
\begin{equation}\label{eq:theta}
 \Theta_{n,m}:\quad
 p_m^n\ \leq\
 R_{n,m}:=
 \bigoplus_{h=0}^{n-1}\ \bigoplus_{i=1}^m
 p_m^h q_{m,i}p_m^{n-1-h}.
\end{equation}
A block $p_m^0$ in this formula is omitted. Every actual monomial has
length $nm>0$, so \eqref{eq:theta} is an identity in the original
constant-free signature.

\begin{lem}\label{lem:blockidentity}
For all $n\geq1$ and $m\geq3$,
\[
 \UT_n(\overline S)\models\Theta_{n,m}.
\]
The same assertion holds when finite entries are allowed to be arbitrary
real numbers, with $-\infty$ still serving as zero.
\end{lem}
\begin{proof}
Let $x_j\mapsto X_j$ be an arbitrary matrix assignment, and fix an
output coordinate. If the value of $p_m^n$ at that coordinate is
$-\infty$, the required inequality holds. Otherwise, choose a path
attaining the finite maximum. All indices on this path
are nondecreasing, so there are at most $n-1$ strict advances.
Divide the word into its $n$ consecutive copies of $p_m$.
At least one entire copy, say the $(h+1)$st, contains no strict advance.
Every factor in that block is therefore a loop at one state $a$.
Let $\alpha_j=(X_j)_{aa}$ be the finite weight of its $x_j$-loop.

There exists an index $i$ with $\alpha_i\geq\alpha_{i^+}$, since the
opposite strict inequalities cannot hold around the entire cycle. Replace that loop-only block by $q_{m,i}$, keeping the rest of
the path unchanged and using only loops at $a$ in the replacement.
The new path belongs to the word
$p_m^h q_{m,i}p_m^{n-1-h}$ and changes the weight by
$\alpha_i-\alpha_{i^+}\geq0$. This word occurs in the right side of
\eqref{eq:theta}. Thus the right side is at least the selected maximum
of the left side at every coordinate.
The proof compares two finite loop weights and does not require them
to be nonnegative.
\end{proof}

\begin{thm}\label{thm:triangularrank}
Fix $n\geq1$. Every ai-semiring variety $\mathcal X$ satisfying
\begin{equation}\label{eq:triangintervalnfb}
 \V(S)\subseteq\mathcal X\subseteq\Uvar_n
\end{equation}
has infinite axiomatic rank. In particular, $\Uvar_n$ is nonfinitely
based, also relative to the ai-semiring axioms.
\end{thm}
\begin{proof}
Let $r\geq1$ be arbitrary, and put $m=\max\{3,r+1\}$.
Consider the finite commutative ai-semiring $C_{n,m}$ given by
Theorem~\ref{thm:algebraiclocality}. Since $S\in\mathcal X$, we have
\[
 \Id_r(\mathcal X)\subseteq\Id_r(S),
 \qquad C_{n,m}\models\Id_r(\mathcal X).
\]
On the other hand, $\Theta_{n,m}$ holds throughout $\mathcal X$
by Lemma~\ref{lem:blockidentity}.

In $C_{n,m}$ assign $x_j=[\{X_j\}]$. Its commutative multiplication
makes the left side of $\Theta_{n,m}$ the class of
$\mu=(X_1\cdots X_m)^n$. For each $h$, the right-hand word
$p_m^h q_{m,i}p_m^{n-1-h}$ becomes the class of
$\nu_i=\mu X_i/X_{i^+}$. By \eqref{eq:residualseparation}, the
left side is $\top$ and the entire right side is $\omega\ne\top$.
Thus $C_{n,m}\not\models\Theta_{n,m}$.

We have therefore found an ai-semiring satisfying every identity in
$\Id_r(\mathcal X)$ but failing an identity of $\mathcal X$.
By soundness of equational logic, $\Id_r(\mathcal X)$ does not imply
all identities of $\mathcal X$. Since $r$ was arbitrary, the variety
has infinite axiomatic rank. Moreover, $C_{n,m}$ satisfies all the
ai-semiring laws, so the same argument applies when those laws are
available as background axioms. This proves the relative assertion
as well.
\end{proof}

\begin{cor}\label{cor:UTallNFB}
Every variety in $[\Uvar_n,\Uvar_m]$, for $1\leq n<m$, has infinite
axiomatic rank and is nonfinitely based.
\end{cor}
\begin{proof}
The map $a\mapsto aI_k$ embeds $S$ into
$\UT_k(\overline S)$: it has diagonal entries $a$ and off-diagonal
entries $-\infty$. Thus $\V(S)\subseteq\Uvar_n$. Apply
Theorem~\ref{thm:triangularrank} with upper endpoint $\Uvar_m$.
\end{proof}

\begin{remark}\label{rem:paddingfails}
\[
 (zxz)(zyz)\leq(z^3xz^3)^2\oplus(z^3yz^3)^2.
\]
In $\UT_2(\overline S)$, set $z=I$ and
\[
 X=\begin{pmatrix}-\infty&0\\-\infty&-\infty\end{pmatrix},\qquad
 Y=\begin{pmatrix}-\infty&-\infty\\-\infty&0\end{pmatrix}.
\]
Then $(XY)_{12}=0$, whereas $(X^2\oplus Y^2)_{12}=-\infty$.
Consequently, this padded identity does not hold in
$\UT_2(\overline S)$. The proof of Theorem~\ref{thm:triangularrank}
uses \eqref{eq:theta} instead.
\end{remark}

\section{The finite basis problem for full matrices}
\label{sec:fullnfb}

We now apply the finite commutative algebras of Section~\ref{sec:algebra}
to full matrices. 
We give its details, including the entrywise bounds
and an explicit matrix example. Combining the estimate with $C_{1,m}$
yields an algebraic proof that does not use the convex-ideal countermodels
or their locality theorem.

\subsection{Cubic padding and uniform entrywise bounds}
The scalar identity $e_m$ holds in $S$, but
Example~\ref{src:ex:scalarfails} prevents its direct use for full
matrices. For every $m\geq2$, use \eqref{eq:scalarwords} and put
\begin{equation}\label{src:eq:paddedblocks}
a_j=zx_jz,\qquad b_j=z^3x_jz^3,
\end{equation}
and define
\begin{equation}\label{src:eq:LR}
L_m=a_1a_2\cdots a_m,\qquad
R_{m,i}=q_{m,i}(b_1,\ldots,b_m),\qquad
R_m=\bigoplus_{i=1}^{m}R_{m,i}.
\end{equation}
The padded identity is
\begin{equation}\label{eq:padding}
\varepsilon_m:\qquad L_m\oplus R_m\approx R_m.
\end{equation}
These are terms in the binary signature. The exponent $3$ means
three repeated factors of $z$, and each product in $R_{m,i}$
retains the increasing variable order fixed in
\eqref{eq:scalarwords}. The identity $\varepsilon_m$ uses $m+1$
distinct variables.

For example, $\varepsilon_2$ states
\[
(zxz)(zyz)\leq
(z^3xz^3)^2\oplus(z^3yz^3)^2.
\]
When $m=3$, its right-hand side is
\[
b_1^2b_3\oplus b_1b_2^2\oplus b_2b_3^2.
\]
These orders matter for matrices; no rearrangement of factors
will be used.

\begin{lem}\label{src:lem:cubic}
Let $Z$ be an $n\times n$ matrix with entries in $\R_{\geq0}$,
and let $M=\max_{u,v}Z_{uv}$. Every entry of $Z^3$ is at least $M$.
\end{lem}

\begin{proof}
Choose $u,v$ with $Z_{uv}=M$. For arbitrary $a,b$, the path
$a\to u\to v\to b$ is one of the candidates in \eqref{src:eq:path}.
It has weight
\[
Z_{au}+Z_{uv}+Z_{vb}\geq M,
\]
because both connecting entries are nonnegative. Hence
$(Z^3)_{ab}\geq M$.
\end{proof}

\begin{remark}\label{src:rem:square}
The corresponding uniform bound need not hold for $Z^2$.
For
\[
Z=\begin{pmatrix}100&0\\0&0\end{pmatrix}
\]
the maximum entry is $100$, but $(Z^2)_{22}=0$.
The three-factor path in Lemma~\ref{src:lem:cubic} allows a specified
largest entry to be connected to arbitrary initial and final
indices.
\end{remark}

\begin{pro}\label{lem:fullpadding}
For every $n\geq1$ and $m\geq2$, the matrix ai-semiring
$\Mat_n(\SR)$ satisfies $\varepsilon_m$. In particular, $\varepsilon_m$
holds in $\M n$.
\end{pro}

\begin{proof}
Fix an arbitrary assignment
\[
z=Z,\qquad x_j=X_j\quad(1\leq j\leq m)
\]
of matrices with finite nonnegative real entries. Put
\[
M=\max_{a,b}Z_{ab},\qquad
\Mat_j=\max_{a,b}(X_j)_{ab},\qquad
H=2mM+\sum_{j=1}^{m}\Mat_j.
\]
Write $A_j=ZX_jZ$ and $B_j=Z^3X_jZ^3$ for the values of the
two types of blocks.

The product $A_1\cdots A_m$ contains exactly $2m$ factors equal
to $Z$ and one factor equal to each $X_j$. Every entry of $Z$
is at most $M$, and every entry of $X_j$ is at most $\Mat_j$.
The upper-bound part of Lemma~\ref{src:lem:paths}, applied to this
list of factors, gives
\begin{equation}\label{src:eq:Lbound}
(L_m(Z,X_1,\ldots,X_m))_{ab}\leq H
\quad\text{for all }a,b.
\end{equation}

Next, choose indices $u_j,v_j$ where $X_j$ has its maximum
entry $\Mat_j$. Since every entry of $Z^3$ is at least $M$, the
particular term with these intermediate indices in
$Z^3X_jZ^3$ gives
\[
(B_j)_{ab}\geq
(Z^3)_{a u_j}+(X_j)_{u_j v_j}+(Z^3)_{v_j b}
\geq 2M+\Mat_j.
\]
Thus
\begin{equation}\label{src:eq:Bbound}
(B_j)_{ab}\geq2M+\Mat_j
\quad\text{for every }a,b\text{ and }j.
\end{equation}
The bound is uniform over all entries; it does not merely
assert that $B_j$ has one large entry.

Let
\[
c_{ij}=1+\mathbf1_{j=i}-\mathbf1_{j=i^+}.
\]
The word $R_{m,i}$ has $c_{ij}$ copies of $b_j$, and
$\sum_jc_{ij}=m$. Applying the lower-bound part of
Lemma~\ref{src:lem:paths} to these $m$ factors, in their specified
order, yields
\begin{align}
(R_{m,i}(Z,X_1,\ldots,X_m))_{ab}
&\geq\sum_{j=1}^{m}c_{ij}(2M+\Mat_j)\notag\\
&=2mM+\sum_{j=1}^{m}\Mat_j+\Mat_i-\Mat_{i^+}\notag\\
&=H+\Mat_i-\Mat_{i^+}.\label{src:eq:Rbound}
\end{align}
This calculation uses the multiplicities of the factors to
add scalar lower bounds; it does not require commutativity
of the matrix product.

There is an index $i$ with $\Mat_i\geq \Mat_{i^+}$, by the same cyclic
argument as in Lemma~\ref{src:lem:scalarvalid}. Choose this single
index for all pairs $a,b$. Equation~\eqref{src:eq:Rbound} gives
\[
(R_m(Z,X_1,\ldots,X_m))_{ab}
\geq(R_{m,i}(Z,X_1,\ldots,X_m))_{ab}
\geq H.
\]
Together with \eqref{src:eq:Lbound}, this proves $L_m\leq R_m$
entrywise. The assignment was arbitrary, so $\varepsilon_m$
holds in $\Mat_n(\SR)$.
\end{proof}

The cubic bound and the comparison with $H$ do not depend on
$n$. This dimension independence will also permit a conclusion
for the variety generated by all matrix dimensions.

\begin{example}\label{src:ex:matrixworked}
Consider the assignment for $m=n=2$
\[
Z=\begin{pmatrix}2&0\\1&0\end{pmatrix},\qquad
X_1=\begin{pmatrix}1&3\\0&2\end{pmatrix},\qquad
X_2=\begin{pmatrix}0&1\\4&0\end{pmatrix}.
\]
Its maxima are $M=2$, $\Mat_1=3$, $\Mat_2=4$, and therefore $H=15$.
The relevant powers and blocks are
\[
Z^2=\begin{pmatrix}4&2\\3&1\end{pmatrix},
\qquad
Z^3=\begin{pmatrix}6&4\\5&3\end{pmatrix},
\]
\[
A_1=ZX_1Z=\begin{pmatrix}6&5\\5&4\end{pmatrix},
\qquad
A_2=ZX_2Z=\begin{pmatrix}6&4\\6&4\end{pmatrix}.
\]
For example,
\[
(A_1A_2)_{12}
=\max\{6+4,5+4\}=10.
\]
The full left-hand product is
\[
L_2=A_1A_2=\begin{pmatrix}12&10\\11&9\end{pmatrix},
\]
whose entries are indeed at most $15$. On the right,
\[
B_1=Z^3X_1Z^3=B_2=Z^3X_2Z^3
=\begin{pmatrix}14&12\\13&11\end{pmatrix},
\]
and hence
\[
R_2=B_1^2\oplus B_2^2
=\begin{pmatrix}28&26\\27&25\end{pmatrix}.
\]
The inequality $L_2\leq R_2$ holds in all four entries.
The estimates \eqref{src:eq:Bbound} give the respective uniform
bounds $7$ and $8$ for $B_1$ and $B_2$; the two right-hand
words consequently have lower bounds $14$ and $16$.
The latter exceeds $H=15$, which is enough for the proof even
though the actual right-hand entries are larger.
\end{example}

\subsection{A transfer criterion and infinite axiomatic rank}
The two ingredients may be isolated in a criterion for a nonempty
class of ai-semirings, rather than just a single matrix algebra.

\begin{pro}\label{pro:fulltransfer}
Let $\mathcal C$ be a nonempty class of ai-semirings in the binary
signature. Suppose that
\[
 \Id(\mathcal C)\subseteq\Id(S)
\]
and that every member of $\mathcal C$ satisfies $\eps_m$ for every
$m\geq3$. Then $\mathcal C$ has infinite axiomatic rank. Its
equational theory has no finite basis, including no finite basis
relative to the ai-semiring axioms.
\end{pro}
\begin{proof}
Let $r\geq1$ be arbitrary and put $m=\max\{3,r+1\}$.
By Theorem~\ref{thm:algebraiclocality}, the finite commutative
ai-semiring $C_{1,m}$ satisfies every $r$-variable identity of $S$.
Hence
\[
 C_{1,m}\models\Id_r(\mathcal C).
\]
In $C_{1,m}$, assign $z$ its actual multiplicative identity
$e=[\{1\}]$ and assign $x_j=[\{X_j\}]$. Both $zx_jz$ and
$z^3x_jz^3$ then have value $[\{X_j\}]$. Consequently, the two
sides of the inequality in \eqref{eq:padding} have the values of
$\eta_{1,m}$, namely $\top$ and $\omega\ne\top$ by
\eqref{eq:residualseparation}. Thus
\[
 C_{1,m}\not\models\eps_m.
\]
By hypothesis, $\eps_m$ belongs to $\Id(\mathcal C)$. Soundness of
equational logic therefore shows that $\Id_r(\mathcal C)$ cannot
imply $\eps_m$. This argument includes arbitrary substitutions in
any proposed derivation: every equational consequence of identities
valid in an algebra is valid in that algebra. Since $r$ was arbitrary,
$\mathcal C$ has infinite axiomatic rank.

If a finite basis existed, choose $r$ bounding the number of variables
in each of its identities. The same countermodel would satisfy the
basis and fail the valid identity $\eps_m$, a contradiction.
The relative assertion follows because $C_{1,m}$ is itself an
ai-semiring and satisfies every background ai-semiring axiom.
\end{proof}

\begin{cor}\label{cor:fullsingle}
For every positive integer $n$, the full max-plus matrix ai-semiring
$\Mat_n(S)$ has infinite axiomatic rank in the signature
$\{\oplus,\odot\}$. Its generated variety has no finite equational
basis, either absolutely or relative to the ai-semiring axioms.
\end{cor}
\begin{proof}
Lemma~\ref{src:lem:embedding} gives
$\Id(\Mat_n(S))\subseteq\Id(S)$, and
Proposition~\ref{lem:fullpadding} gives every $\eps_m$.
Apply Proposition~\ref{pro:fulltransfer} to the singleton class
$\{\Mat_n(S)\}$.
\end{proof}

\begin{remark}\label{src:rem:rank}
The parameter $m=\max\{3,r+1\}$ depends on the variable bound $r$,
not on the matrix dimension. A countermodel may be commutative and
need not be a matrix algebra: an equational implication must hold in
every algebra satisfying its premises. 

The conclusion is stronger than nonfinite basability. For every
fixed variable bound, a valid identity is not implied by all valid
identities within that bound. We do not assert that the padded
identities $\eps_m$ form a complete basis, or that they are pairwise
independent. The independent identities used for the continuum
intervals are the different family $\sigma_{n,p}$.
\end{remark}

\begin{thm}\label{thm:fullrank}\label{src:cor:interval}
Let
\[
 \mathcal V_0=\V(S),\qquad
 \mathcal W=\Fvar_\infty=\bigvee_{k\geq1}\Fvar_k
 =\V\bigl(\{\Mat_k(S):k\geq1\}\bigr).
\]
Every variety $\mathcal X$ with
$\mathcal V_0\subseteq\mathcal X\subseteq\mathcal W$
has infinite axiomatic rank and no finite equational basis, also
relative to the ai-semiring laws. In particular, the assertion holds
for every variety in $[\V(S),\V(\Mat_n(S))]$ for each fixed $n$.
\end{thm}
\begin{proof}
Since $S\in\mathcal X$, every identity of $\mathcal X$ holds in $S$.
By Proposition~\ref{lem:fullpadding}, $\eps_m$ holds in every
generator of $\mathcal W$, for every $m\geq3$. Identities are
preserved by homomorphic images, subalgebras and direct products,
so all these identities hold throughout $\mathcal W$ and hence
throughout $\mathcal X$. Apply Proposition~\ref{pro:fulltransfer}.

For fixed $n$, Lemma~\ref{src:lem:embedding} gives
$\mathcal V_0\subseteq\V(\Mat_n(S))$, and the definition of
$\mathcal W$ gives $\V(\Mat_n(S))\subseteq\mathcal W$.
Thus the fixed-dimension interval is contained in the interval just
treated.
\end{proof}

Theorem~\ref{thm:fullrank} concerns only subvarieties containing
$\V(S)$. It does not classify all subvarieties of $\V(\Mat_n(S))$ or
determine their number by itself. Those cardinality assertions require
the separate construction in Section~\ref{sec:continuum}.
Only validity of the padded identities in $\mathcal W$ has been
used; no assertion that they axiomatize $\mathcal W$ is needed.

\begin{proof}[Proof of Theorem~\ref{thm:main}]
Strictness follows from Section~\ref{sec:matrix}. The order embeddings
and cardinalities are proved in Theorem~\ref{thm:intervals} and
Corollary~\ref{cor:cardinality}. Constant matrices embed $S$ into every
$\Mat_n(S)$, so Theorem~\ref{thm:fullrank} applies to every variety in a
full-matrix interval. Corollary~\ref{cor:UTallNFB} treats every variety
in an upper triangular interval. The assertions about chains,
antichains, and relative nonfinite basability have already been included
in those results.
\end{proof}

\section{Coefficient domains and equational equivalence}\label{sec:domains}

The first part concerns full matrices with finite
nonnegative entries. The second part records the different conclusions
that remain valid for more general tropical coefficient domains.

\subsection{Six equivalent full-matrix theories}
The preceding matrix proof works directly for nonnegative real
coefficients, and the scalar identity inclusion remains available
over larger domains.

\begin{cor}\label{src:cor:nonnegative}
For every $n\geq1$, both $\Mat_n(S_{\Q_{\geq0}})$ and
$\Mat_n(S_{\R_{\geq0}})$ have infinite axiomatic rank and have
no finite equational basis, either absolutely or relative
to the ai-semiring axioms.
\end{cor}

\begin{proof}
Each matrix algebra contains the constant matrices with
nonnegative integer entries and thus contains a copy of
$\Smax$. Its identities are therefore contained in
$\Id(\Smax)$. Proposition~\ref{lem:fullpadding} supplies all
$\varepsilon_m$, and Proposition~\ref{pro:fulltransfer} applies.
\end{proof}

The same conclusion can be extended to positive entries by
comparing the complete equational theories. In the following
argument, multiplying all entries by $c>0$ means ordinary
arithmetic scaling.

\begin{lem}\label{src:lem:termfunctions}
Fix $n$ and a binary-signature term $t$ in $k$ variables.
For each pair $a,b$, the entry
$t(X_1,\ldots,X_k)_{ab}$ is a finite maximum of ordinary
homogeneous linear forms in the $kn^2$ input entries, with
nonnegative integer coefficients. It is a continuous function
on $\R^{kn^2}$ and satisfies
\begin{equation}\label{src:eq:homogeneity}
t(cX_1,\ldots,cX_k)=c\,t(X_1,\ldots,X_k)
\quad(c>0).
\end{equation}
\end{lem}

\begin{proof}
An entry of a variable is itself a linear form. If the claim
holds for $s$ and $u$, an entry of $s\oplus u$ is the maximum
of the combined finite lists of forms. An entry of $su$ is
\[
\max_h\bigl(s(X_1,\ldots,X_k)_{ah}
             +u(X_1,\ldots,X_k)_{hb}\bigr).
\]
For each $h$, the sum of two finite maxima is the maximum of
all pairwise sums of their forms. These sums are again
homogeneous linear forms with nonnegative integer coefficients.
Induction on the term proves the representation. Finite
maxima of continuous linear functions are continuous. Every
such form scales by $c$, proving \eqref{src:eq:homogeneity}.
\end{proof}

\begin{pro}\label{src:prop:domains}
For each fixed $n\geq1$, the matrix ai-semirings over the six
coefficient domains
\[
\N,\quad\mathbb N_{>0},\quad
\Q_{\geq0},\quad\Q_{>0},\quad
\R_{\geq0},\quad\R_{>0}
\]
have exactly the same identities in the binary signature.
\end{pro}

\begin{proof}
Suppose that an identity $s\approx t$ holds for matrices with
positive integer entries. We first prove that it holds for
positive rational entries. Given any assignment of positive
rational matrices, choose a common positive integer
denominator $d$ for all entries in all input matrices.
Scaling every input by $d$ produces positive integer
matrices. The assumed identity and \eqref{src:eq:homogeneity} give
\[
d\,s(X_1,\ldots,X_k)
=s(dX_1,\ldots,dX_k)
=t(dX_1,\ldots,dX_k)
=d\,t(X_1,\ldots,X_k).
\]
Dividing ordinary matrix entries by $d$ proves the desired
equality for the rational assignment.

Next, positive rational assignments are dense in the space
of positive real assignments. For an arbitrary positive real
assignment, approximate all input entries simultaneously by
positive rationals. Lemma~\ref{src:lem:termfunctions} permits passage
to the limit in every output entry, proving $s\approx t$ for
positive real matrices.

Finally, let $X_1,\ldots,X_k$ have nonnegative real entries,
and let $J$ denote the ordinary matrix all of whose entries
are $1$. For $\eta>0$, the matrices $X_j+\eta J$, where the
addition here is ordinary entrywise addition, have positive
real entries. The identity holds on these matrices. Taking
$\eta\downarrow0$ and using continuity proves it for the
original nonnegative assignment. We have shown
\[
\Id(\Mat_n(S_{\mathbb N_{>0}}))
\subseteq\Id(\Mat_n(S_{\R_{\geq0}})).
\]

Every one of the six domains lies between $\mathbb N_{>0}$
and $\R_{\geq0}$ by inclusion. If $D_1\subseteq D_2$,
the entrywise inclusion is an embedding of
$\Mat_n(S_{D_1})$ into $\Mat_n(S_{D_2})$, and identities have
the reverse inclusion. Hence, for any of the six domains $D$,
\[
\Id(\Mat_n(S_{\R_{\geq0}}))
\subseteq\Id(\Mat_n(S_D))
\subseteq\Id(\Mat_n(S_{\mathbb N_{>0}})).
\]
The inclusion proved above makes all three sets equal,
establishing the assertion.
\end{proof}

\begin{cor}\label{src:cor:positive}
For each $n\geq1$ and each of the six coefficient domains
in Proposition~\ref{src:prop:domains}, the full matrix
ai-semiring has infinite axiomatic rank and no finite
equational basis, also relative to the ai-semiring axioms.
\end{cor}

\begin{proof}
By Proposition~\ref{src:prop:domains}, each algebra has the same
equational theory as $\M n$. The assertions therefore follow
from Corollary~\ref{cor:fullsingle}.
\end{proof}

\begin{remark}\label{src:rem:signature}
The scaling argument concerns ordinary degree-one homogeneity
of the entry functions in all input entries. It does not
require that the words occurring in a term have equal
lengths. For instance, the scalar term $x\oplus x^2$ has
value $\max\{a,2a\}$ and still scales by $c$ when $a$ is
replaced by $ca$. The absence of named constants ensures
the inductive homogeneity statement. Results for an expanded
signature with constants would require a separate argument.
\end{remark}

\begin{remark}\label{src:rem:scope}
The padded identities are sensitive to the coefficient domain.
If negative entries are allowed, already in dimension one
the assignment $z=-1$, $x_1=x_2=0$ gives
\[
L_2=-4,\qquad R_2=-12.
\]
Thus $\varepsilon_2$ fails. Here each block $zx_jz$ has
ordinary value $-2$, while each block $z^3x_jz^3$ has value
$-6$. This also indicates exactly where the nonnegative
connecting entries in Lemma~\ref{src:lem:cubic} are needed.

Allowing $-\infty$ produces a different obstruction.
In the usual tropical matrix semiring, let
\[
I=\begin{pmatrix}0&-\infty\\-\infty&0\end{pmatrix}
\]
and assign $z=I$. Then $I^3=I$ and all padding disappears.
For the matrices $A,B$ from Example~\ref{src:ex:scalarfails},
the identity $\varepsilon_2$ becomes
$AB\leq A^2\oplus B^2$, which is false. Thus the same family
of separating identities cannot be used in that setting.
These failures delimit the cubic-padding argument; they do not by
themselves decide the finite basis problem over coefficient domains
with negative entries or $-\infty$. In the upper triangular case,
Lemma~\ref{lem:blockidentity} supplies the different identities
$\Theta_{n,m}$, and Theorem~\ref{thm:triangularrank} supplies the
separate nonfinite basis proof. The explicit triangular failure of
$\eps_2$ is given in Remark~\ref{rem:paddingfails}.
\end{remark}

\subsection{Upper triangular and more general coefficient domains}

The main results concern $\overline S$, whose finite entries are
nonnegative integers. The same constructions yield the following
extension. It does not identify the equational theory of $\overline S$
with that of the real tropical semiring.

\begin{cor}\label{cor:otherdomains}
Let $D\subseteq\R$ be closed under ordinary addition and maximum and
contain $\N$. Set $D_{\max}=(D,\max,+)$ and
$\overline D_{\max}=D_{\max}\cup\{-\infty\}$.
For all $1\leq n<m$,
\[
 \bigl|[\V(\UT_n(\overline D_{\max})),
          \V(\UT_m(\overline D_{\max}))]\bigr|=2^{\aleph_0},
\]
and every variety in this interval has infinite axiomatic rank.
In particular, the conclusion applies to $D=\Z,\Q,\R$, as well as
to the nonnegative integer, rational, and real domains.

For full matrices, the same continuum conclusion holds for the interval
\[
 [\V(\Mat_n(D_{\max})),\V(\Mat_m(D_{\max}))].
\]
When $D$ contains negative numbers, no full-matrix nonfinite-basis
claim is inferred here from the nonnegative-padding argument.
\end{cor}
\begin{proof}
The finite algebras $\Alg_{n,p}$ already belong to
$\V(\UT_{n+1}(\overline S))$, which is contained in the corresponding
variety over $\overline D_{\max}$. The lower endpoint still satisfies
\eqref{eq:AL}, since the coefficient multiplication is commutative.
Thus the identical joins and identities from
Theorem~\ref{thm:intervals} give the continuum lower bound; the upper
bound is unchanged. Strictness is witnessed by the same weight-$0/1$
matrices.

The scalar algebra $S$ embeds in each triangular algebra over
$\overline D_{\max}$, and Lemma~\ref{lem:blockidentity} remains valid
for real finite weights. Therefore the proof of
Theorem~\ref{thm:triangularrank} applies without change.
For full matrices use
$\Alg_{n,p}\in\V(\Mat_{n+1}(S))\subseteq\V(\Mat_{n+1}(D_{\max}))$
and Lemma~\ref{lem:AL} to obtain the same order embedding. The last
sentence records the nonnegativity hypothesis required in
Proposition~\ref{lem:fullpadding}.
\end{proof}

\bibliographystyle{amsplain}

\end{document}